\documentclass{amsart}
\usepackage{ifthen}
\usepackage{amssymb,amsmath}
\usepackage{enumerate}
\usepackage[colorlinks=true]{hyperref}

\numberwithin{equation}{section}

\def\phi{\varphi}
\def\R{\mathbb R}
\def\N{\mathbb N}

\newcommand{\bfR}{\mathbb{R}}
\newcommand{\bfS}{\mathbb{S}}

\newcommand{\e}{\varepsilon}
\newcommand{\goto}{\rightarrow}
\newcommand{\ol}{\overline}

\newcommand{\be}{\begin{equation}}
\newcommand{\ee}{\end{equation}}
\newcommand{\bea}{\begin{eqnarray}}
\newcommand{\eea}{\end{eqnarray}}

\newtheorem{theorem}{Theorem}[section]
\newtheorem{lemma}[theorem]{Lemma}
\newtheorem{corollary}[theorem]{Corollary}

\theoremstyle{definition}
\newtheorem{definition}[theorem]{Definition}
\newtheorem{example}[theorem]{Example}

\theoremstyle{remark}
\newtheorem{remark}[theorem]{Remark}
\numberwithin{equation}{section}
\def\fuhome{@math.fu-berlin.de}
\def\impahome{@impa.br}
\def\jhuhome{@math.jhu.edu}

\begin{document}

\title{ The half-space property and entire positive minimal graphs in $M \times \R$.}
\author{Harold Rosenberg}
\thanks{}
\address{Harold Rosenberg: Insituto Nacional de Mathem\'atica Pura e
  Applicada (IMPA), Estrada Dona Castorina 110, 22460 Rio de Janeiro -
RJ, BRAZIL}
\curraddr{}
\email{rosen\impahome}

\author{Felix Schulze}
\thanks{}
\address{Felix Schulze: 
  Freie Universit\"at Berlin, Arnimallee 3, 
  14195 Berlin, GERMANY}
\curraddr{}
\email{Felix.Schulze\fuhome}

\author{Joel Spruck}
\thanks{Research of the third author supported in part by the NSF and Simons Foundation.}
\address{Joel Spruck: 
  Department of Mathematics, John Hopkins University,
  Baltimore, MD 21218, USA}
\curraddr{}
\email{js\jhuhome}

\subjclass[2000]{}

\dedicatory{}


\begin{abstract}
 We show that a properly immersed minimal hypersurface in \mbox{$M\times
 \R_+$} equals some $M\times\{c\}$ when $M$ is a complete,
recurrent $n$-dimensional Riemannian manifold with bounded curvature.
If on the other hand, $M$ has nonnegative Ricci curvature with curvature bounded below,
the same result holds for any positive entire minimal graph over $M$.
\end{abstract}

\maketitle
\section{Introduction}
\label{sec1}
A problem that has received considerable attention is to give
conditions which force two minimal submanifolds $S_1, \,S_2$ of a
Riemannian manifold $N$ to intersect. If they do not intersect,
does this determine the geometry of $S_1, \,S_2$ in $N$?

Perhaps the simplest example of this situation is when $N$ is a
strictly convex ovaloid (i.e an $\bfS^2$ with a metric of positive
curvature) and $S_1, \,S_2$ are complete embedded geodesics of $N$.
There is a three dimensional version of this simple example. Let $N$
be a compact 3-dimensional manifold with positive sectional
curvatures.  Then if $S_1, \,S_2$ are finite topology complete minimal
surfaces embedded in $N$, they must intersect. This follows from the
minimal lamination closure theorem \cite{MR}.  There is also the
classical theorem of Frankel \cite{F} which states that if $N$ be a
closed n dimensional manifold with positive Ricci curvature and $S_1,
\,S_2$ are compact minimal $(n-1)$ dimensional submanifolds immersed
in $N$, then they intersect. For some other results on this problem, 
 see \cite{DH}, \cite{DMR}, \cite{Mazet_10}, \cite{ER}, \cite{HRS_08}.

In this paper we consider this question when $N = M\times \R$ where
$M$ is a complete n dimensional Riemannian manifold, $S_1 = M \times \{0\}$
and $S_2$ is a properly immersed minimal hypersurface in $M\times
\R_+$. Our problem then becomes to determine what
conditions  on $M$ imply that $S = S_2$ is the totally geodesic slice $M \times \{c\}$ for some positive $c$?

Perhaps the first result in this direction was the celebrated theorem
of Bombieri, De Giorgi and Miranda \cite{BGM} who proved that an
entire minimal positive graph over $\R^n$ is a totally geodesic
slice. The hyperbolic plane $\mathbb{H}^2$ does not have this
property; there are entire bounded minimal graphs that are not slices.

For a proper immersed minimal surface $S$ in $\R^3=\R^2 \times
\R_+$, the foundational result was discovered by Hoffman and Meeks
\cite{HM} who proved that $S = \R^2 \times \{c\}, \, c\geq 0$.
They called this the {\em half-space theorem}. \\

\begin{definition} \label{def1.1} We will say that {\em M has the
    half-space property} if a minimal hypersurface S properly immersed
  in $M \times \R_+$, equals a slice $M \times \{c\}$. Since there are
  rotationally invariant minimal hypersurfaces in $\R^{n+1},\, n > 2$,
  that are bounded above and below (catenoids), $M = \R^n,\, n>2$ does
  not have the half space property but entire minimal positive graphs
  over $\bfR^n$ are slices.
\end{definition}

Hence it is interesting to find conditions on $M$ which ensure that $M$
has the half space property or the property that positive entire
minimal graphs over $M$ are slices.
Our contributions to these questions are the  following two theorems.
 
\begin{theorem} \label{th1.1} Let $M^n$ be a complete recurrent
  Riemannian manifold with bounded sectional curvatures $|K_{\pi}|
  \leq K_0$ for some constant $K_0$.  Then M has the half
  space property.
 \end{theorem}
 
 \begin{theorem} \label{th1.2} Let $M^n$ be a complete Riemannian
   manifold with nonnegative Ricci curvature and sectional curvatures
   $K_{\pi} \geq -K_0$ for a nonnegative constant $K_0$.  Let $S$ be an
   entire minimal graph in $M \times \bfR $ with height function
   $u\geq 0$. Then $S=M\times \{c\}$ for some constant $c\geq 0$.
\end{theorem}

In the same spirit, an interesting question is to study
those complete embedded minimal hypersurfaces in $M \times \R$, whose
angle function $\langle N, \frac{\partial}{\partial t}\rangle$ does
not change sign;  see \cite{ER}.

\begin{definition} \label{def1.2} $M$ in Theorem \ref{th1.1} is recurrent means that
for any nonempty bounded open set  $U$, 
every bounded harmonic function on $M\setminus U$ 
is determined by its boundary values. Furthermore, 
if $M\setminus U$ is quasi-isometric to $N\setminus V$,
then $M$ is recurrent if and only if $N$ is recurrent.  
For a detailed discussion see \cite{Gr2, Li_12}.
\end{definition}

\begin{example} Some interesting examples of allowable $M$ may be
  constructed as follows.  Let $N$ be a closed manifold and take $M = N
  \times \R^2$, or $M = N\times \R$, or $M = N\times S$, $S$ a
  complete surface with quadratic area growth or finite total curvature.
  These examples have quadratic volume growth so they are recurrent.
  Thus removing a bounded non-empty open set from $M$, then
  what is left is parabolic, i.e any bounded harmonic function is
  determined by its boundary values.
\end{example}

\section{Local formulas for minimal graphs}
\label{sec2}
Let $u$ be the height function of an $n$ dimensional minimal graph
$S=\{(x,u(x)): x\in B_R(p)\}$ in $M^n \times \R$ where $M$ is complete with
nonnegative Ricci curvature and $B_R(p)$ is a geodesic ball of radius
$R$ about $p$.  If $ds^2=\sigma_{ij}dx_i dx_j$ is a local Riemannian
metric on $M$, then $M\times\R$ is given the product metric $ds^2+dt^2$
where $t$ is a coordinate for $\R$.  Then the height function $u(x)\in
C^2(\Omega)$ satisfies the divergence form equation
\begin{equation}\label{eq1}
\text{div}^M\bigg(\frac{\nabla^M u}{\sqrt{1+|\nabla^M u|^2}}\bigg) =
0\,
\end{equation}
where the divergence and gradient $\nabla^Mu$ are taken with
respect to the metric on $M$.  Equivalently, equation (\ref{eq1}) can be
written in non-divergence form 
\be \frac1{W}g^{ij}D_i D_j u=0\, , \qquad \text{where}\ W=\sqrt{1+|\nabla^M
  u|^2}\, ,
\label{eq2} \ee
$D$ denotes covariant differentiation on M and 
\[g^{ij}=\sigma^{ij}-\frac{u^i u^j}{W^2}\;,\; u^i=\sigma^{ij}D_j u\, .\]

This can be seen as follows.
 Let $ x_1,\ldots x_n$ be a system of local coordinates for M with corresponding metric $\sigma_{ij}$.
Then the coordinate vector fields for S 
and the upward unit normal  to S is given by
\be
X_i=\frac{\partial}{\partial x_i}+u_i \frac{\partial}{\partial t}
\label{eq3}
\ee
and
\be
N=\frac1{W}\Big(-u^j \frac{\partial}{\partial x_j}+
\frac{\partial}{\partial t}\Big)\, ,\ u^i=\sigma^{ij}u_j~.
\label{eq4}
\ee
The induced metric on S is then 
\be
g_{ij}=\langle X_i,X_j \rangle=\sigma_{ij}+u_iu_j  
\label{eq5}
\ee
with  inverse
\be
g^{ij}=\sigma^{ij}-\frac{u^i u^j}{W^2}~. 
\label{eq6}
\ee
It is easily seen that 
\be
g=\det(g_{ij})=\sigma W^2~,~\sigma=\det(\sigma_{ij})~. 
\label{eq7}
\ee

The second fundamental form $b_{ij}$ of S is given by ($\overline{D}$ is covariant differentiation
on $M\times \R$)
\be \begin{split}
b_{ij}&= \big\langle\overline{D}_{X_i}X_j,\nu\big\rangle
= \Big\langle D_{\frac{\partial}{\partial x_i}}\frac{\partial}{\partial x_j}+
u_{ij}\frac{\partial}{\partial t}, N \Big\rangle\\
&= \Big\langle\Gamma^k_{ij}\frac{\partial}{\partial x_k}+u_{ij}\frac{\partial}{\partial t},\nu\Big\rangle
=\frac1W\Big(-\Gamma^k_{ij}u^{l}\sigma_{kl}+u_{ij}\Big)\, .
\end{split}
\ee

Hence,
\be
b_{ij}=\frac{D_iD_ju}{W}
\label{eq8}
\ee
and so the  mean  curvature $H$ of $S$ is then given by
\be
nH=\frac1W g^{ij}D_i D_j u~.
\label{eq9}
\ee

The area functional of $S$ is given in local coordinates by
\[A(S)=\int W  \sqrt{\sigma}~dx~.\]
As a functional of $u$, this gives the 
Euler-Lagrange equation
\be \label{eq9.5}
 \text{div}^M\bigg(\frac{Du}{W}\bigg)=\frac1{\sqrt{\sigma}}D_i\bigg(\sqrt{\sigma}\frac{u^i}{W}\bigg)=0\,
 .
\ee
It is easily seen that (\ref{eq2}) is the non-divergence form of (\ref{eq9.5}).\\

We will also need the  well known formulae 
\begin{align}
\label{eq10} \Delta^S\, u&=0\\
\Delta^S\, W^{-1}&=-\big(|A|^2+\widetilde{\text{Ric}}(N,N)\big) W^{-1}\, ,
\label{eq11}
\end{align}
where $|A|$ is the norm of the second fundamental form of $S$, $\widetilde{\text{Ric}}$ is the Ricci curvature of $M\times \R$,
and $\Delta^{S}$ is the
Laplace-Beltrami operator of $S$ given in local coordinates by
\be
\Delta^S\equiv \text{div}^{S}\big(\nabla^{S}\cdot\big)=\frac1{\sqrt{g}}D_i\big(\sqrt{g}g^{ij}D_j\cdot\big)=g^{ij}D_i D_j~.
\label{eq12}
\ee

Since $\tau:=\frac{d}{dt}$ is a Killing vector field on $M\times \R,\,
W^{-1}=\langle N,\tau \rangle$ is a Jacobi field and so satisfies the
Jacobi equation \eqref{eq11}.  For a clean derivation of (\ref{eq11})
using moving frames see \cite[section 2]{SY} where M is three
dimensional but the derivation is valid in all dimensions. Equation
(\ref{eq10})
is easily seen to be equivalent to (\ref{eq2}).\\

From \eqref{eq12} follows the important 
formulae
\be
\Delta^S \phi(x)=g^{ij}D_iD_j~\phi
\label{eq12.1}
\ee
and
\be
\Delta^{S}~g(\phi)=g'(\phi)\Delta_S~\phi +g''(\phi)g^{ij}D_i\phi D_j \phi~.
\label{eq12.2}
\ee

This implies that

$$\Delta^S W = 2W^{-1}|\nabla^S W|^2 + W
\big(|A|^2+\widetilde{\text{Ric}}(N,N)\big)\, .$$

Let us for the moment assume that at a point $p \in S$ the normal 
$N$ is not equal to $\tau$. We let
$$\gamma := \frac{p^{TM}(N)}{|p^{TM}(N)|}\, ,$$
where $p^{TM}$ is the projection to the tangent space of the
horizontal plane through $p$ in $M\times \R$. It then holds that
$$ \widetilde{\text{Ric}}(N,N) = \text{Ric}^M(p^{TM}(N),p^{TM}(N))
= \big(1-W^{-2}\big)\text{Ric}^M(\gamma,\gamma)\, .$$
Noting that this is still trivially true if $N=\tau$, we
arrive at
\be \label{eq13}
\Delta^S W = 2W^{-1}|\nabla^S W|^2 + W
|A|^2+W\big(1-W^{-2}\big)\text{Ric}^M(\gamma,\gamma)~.
\ee

Now let $h(x)=\eta(x) W(x)$ with $\eta \geq 0$ smooth. 
Then using \eqref{eq13}, a simple computation gives

\begin{align} \label{eq13.1}
\nonumber 
Lh &:= \Delta^S h-2g^{ij}\frac{D_i W}{W}D_j h \\
&=\eta\big(\Delta^S W-\frac2W g^{ij}D_i W D_j W\big)+
W\Delta^S \eta\\ \nonumber
&= W\big(\Delta^S \eta + \eta \big(|A|^2+(1-W^{-2})\text{Ric}^M(\gamma,\gamma)\big)\big)~.
\end{align}

\section{The recurrent case}
\label{sec3}
The original proof by Hoffman-Meeks  of the half-space theorem in $\R^3$
used the family of minimal surfaces obtained from a catenoid by
homothety. We will use a discrete family of minimal graphs in $M\times
\R$, like the catenoids in $\R^3$. 

Let $D_1 \subset M$ be open and bounded with $\partial D_1$
smooth. Since $M$ has bounded sectional curvatures, we can apply
Theorem 0.1 of Cheeger and Gromov \cite{CheegerGromov_91} to assert
the existence of an exhaustion of $M$, $D_1\subset D_2\subset \cdots
\subset D_n\subset \cdots$ by domains with smooth boundaries, such
that the norm of the second fundamental form of the boundaries
$\partial D_i$ is uniformly bounded by $C_1$ and $\bar D_i \subset
D_{i+1}$. We denote $\partial D_n$ by $\partial_n$ and by $A_n$ the
annular-type domain $D_n\setminus \ol{D_1}$, with $\partial A_n
= \partial_1 \cup \partial_n$.

$A_n$ is a stable minimal hypersurface of $M\times \R$ (it is totally
geodesic) so any sufficiently small smooth perturbation of $\partial
A_n$ to $\Gamma_{n,t}$ gives rise to a smooth family
of minimal hypersurfaces $S_{n,t}$ with $\partial
S_{n,t}=\Gamma_{n,t}$, and $S_{n,0} = A_n$. The $S_{n,t}$
are smooth up to their boundary (we will use $C^2$).

We apply this to the deformation of $\partial A_n$ which is the graph
over $\partial A_n$ given by $\partial_1 \cup (\partial_n \times
\{t\})$, for $t\geq 0$. 
Then for $t$ sufficiently small, $S_{n,t}$ is the graph of a
function smooth $u_{n,t}$ defined on $A_n$, with boundary values $0$ on
$\partial_1$ and $t$ on $\partial_n$. Note that $u_{n,t}$ satisfies
the minimal surface equation
on $A_n$ and  by the maximum principle we have $0 \leq u_{n,t}\leq t$.
Furthermore, as long as
$|\nabla^M u_{n,t}|$ is uniformly bounded, the DeGiorgi-Nash-Moser and
Schauder estimates imply uniform estimates for all higher derivatives up to the
boundary. Thus to apply the method of continuity,  we  need only show
uniform gradient estimates. 

We will first present a maximum principle for the function 
$$ W =   \sqrt{1+|\nabla^M u|^2}$$
on $S=\text{graph}(u)\subset M \times \R$, where we assume that
$u:\Omega\rightarrow \R$ is a solution to the minimal surface equation
on $\Omega \subset M$. From \eqref{eq13}, we see that if the Ricci
curvature of $M$ is nonnegative then $W$ is bounded on $S$ by its
maximum on $\partial S$. To treat the case that the Ricci curvature of
$M$ is only bounded from below we consider the function
$$ h= \eta \cdot W\,,\ \  \eta=e^{\alpha u} $$
where $\alpha>0$. From \eqref{eq10}, \eqref{eq12.2}
we find
$$\Delta^S \eta=\alpha^2 \eta |\nabla^S u|^2=\alpha^2(1-W^{-2})\eta~.$$
Then using \eqref{eq13.1} we have
\begin{equation} \label{eq13.2}
  Lh= h\Big(|A|^2 +  \big(1-W^{-2}\big)\big(\alpha^2+\text{Ric}^M(\gamma,\gamma)\big)\Big) \, .
\end{equation}
This implies the following estimate.
\begin{lemma}\label{lem3.1}
Let $\Omega \subset M$ be open and bounded and let $u\in C^2(\Omega)\cap C^1(\bar{\Omega})$ be
a solution of the minimal surface equation in $\Omega$. Then
$$\sup_{\Omega}\sqrt{1+|\nabla^M u|^2} \leq \sup_{\Omega}e^{-\alpha u}
\cdot \sup_{\partial \Omega}\bigg(e^{\alpha u} \sqrt{1+|\nabla^M
  u|^2}\bigg)\, , $$
where $\alpha^2 = \sup\{\max\{-\text{Ric}^M(\gamma,\gamma), 0\}\, |\, \gamma
  \in T_pM,\, |\gamma|=1, \, p \in \Omega \}$. 
\end{lemma}
\begin{proof} By our choice of $\alpha>0$,  we see from \eqref{eq13.1} that
  $Lh \geq 0$. The result now follows from the maximum principle.
\end{proof}
\begin{remark}
  In the case that $S$ has constant mean curvature $H$ one can
  compute that
$$  Lh= h\Big(\alpha \frac{nH}W+|A|^2 +  \big(1-W^{-2}\big)\big(\alpha^2+\text{Ric}^M(\gamma,\gamma)\big)\Big) \, .
    $$
By considering $-u$ instead of $u$  if necessary, 
 we can assume that $H\geq 0$ and  arrive at the same
gradient estimate as before.
\end{remark}

Lemma \ref{lem3.1} implies that to use the method of continuity for the
surfaces $S_n(t)$ we only need a priori gradient bounds on
$\partial A_n$. 

For convenience of notation, assume the sectional curvatures of $M$
are bounded from above by $K_0=1$. Then the Riccati comparison
estimates imply that for any point $p$ in $M$, the exponential map
$\text{exp}_p:T_pM\supset B_\pi(0) \rightarrow B_{\pi }(p)$ is a local
diffeomorphism. Let us for the moment also assume that the injectivity
radius of $M$ is greater or equal to $1$, i.e. the exponential map
$\text{exp}_p:T_pM\supset B_1(0) \rightarrow B_1(p)$ is actually a
diffeomorphism.

We now almost explicitly construct a catenoid like supersolution $w=w(r;r_0,p)$
of the minimal surface equation in an annulus $A(p):= B_{4r_0}(p)
\setminus B_{2r_0}(p)$
of height $2\delta_0$ where $r=d(x,p)$ is the distance function from x to p.
Here $r_0$ will be chosen sufficiently small depending on the  bound $K_0=1$ for sectional
curvature of M and the lower bound 1 for the  injectivity radius of M.

\begin{lemma} \label{lem3.2}
For $r_0$ sufficiently small, there exists
$w=\phi(r)-\phi(2r_0)$ satisfying 
\begin{align}
&\text{div}^M\bigg(\frac{\nabla^M w}{\sqrt{1+|\nabla^M w|^2}}\bigg) <0
\ \ \text{in}\ A(p)\\
&w=0 \hspace{.1in} \text{on $r=2r_0$}\\
&w=2\delta_0:=\phi(4r_0)-\phi(2r_0) \hspace{.1in} \text{on $r=4r_0$}
\end{align}
where $\phi'(r)>0,\, \phi(r_0)=0,\, \phi'(r_0)=+\infty$ and the inverse function $ r=\gamma(s)$  of $\phi(r)$ is implicitly defined  by
\be \label{eq3.05}
s=\int_{r_0}^{\gamma} \frac{dt}{\sqrt{(\frac{t}{r_0})^{2n}-1}}
\ee
\end{lemma}

\begin{proof} From \eqref{eq2} it suffices to show that in $A(p)$
  \be \label{eq3.10} Mw:=\Big(\sigma^{ij}-\frac{w^i w^j}{W^2}\Big)D_i
  D_j w<0\, ,\ \ \text{ where}\ W=\sqrt{1+|\nabla^M w|^2}\, .  \ee
  When $w=\phi(r)$ we easily find from \eqref{eq3.10} that
  \be \label{eq3.20} Mw= \phi'(r) \Delta^M r +
  \frac{\phi''(r)}{1+\phi'^2(r)} \ee We fix $ r_0$ small enough that
  $\Delta^M r <\frac{n}r$ in $B_{4r_0}(p)$. Then from \eqref{eq3.20},
  \be\label{eq3.30} Mw< \frac{\phi''(r)}{1+\phi'^2(r)}+\frac{n}r
  \phi'(r) \ee and it suffices to solve \be \label{eq3.40}
  \frac{\phi''}{1+\phi'^2}+\frac{n}r \phi'=0 \ee But \eqref{eq3.40} is
  the ode for the height function of the top half of the catenoid in
  $\bfR^{n+1} \times \bfR$ over $\{r>r_0\} \subset \bfR^{n+1}$ and its
  solution is well known to be given as described.
\end{proof}

\begin{remark}\label{rem3.1} Using the continuity method, it is immediate that we  can deform $w$ to an exact solution of the minimal surface equation in $A(p)$.
\end{remark}

We now use the barrier $Z_{r_0,p}=\text{graph}(w)$ near the boundary
of $A_n$ to obtain a gradient bound for $S_n(t)$, provided $0 \leq
t\leq \delta_0$.  Let $p_0 \in \partial A_n$. Since the norm of the
second fundamental form of each component $\partial A_n$ is bounded by
$C_1$ there is a $p_1\in M$ such for $r_0$ sufficiently small
depending only on $C_1$, $B_{2r_0}(p_1)$ touches $A_n$ from the
outside at $p_0$. Note that $B_{2r_0}(p_1)$ still might intersect
$A_n$, but it touches $A_n$ in $p_0$ from the outside. We now consider
the part of $Z_{r_0,p_1}$ which is a graph over the connected
component of $(B_{4 r_0}(p_1)\setminus B_{2 r_0}(p_1))\cap A_n$ which
has $p_0$ in its boundary. Suppose first $p_0 \in \partial_1$. Note
that on its boundary $Z_{r_0,p_1}$ always lies above $S_{n,t}$, as
long as $0\leq t \leq \delta_0$. By the maximum principle this implies
that $Z_{r_0,p_1}$ lies above $S_n(t)$, which in turn implies a
gradient bound for $u_{n,t}$ at $p_0$. By reflecting $Z_{r_0,p}$ at
the plane of height $0$ in $M\times \R$ and translating up by $t$, we
can do a similar construction at the outer boundary $\partial_n$ of
$A_n$ for $S_{n,t}$ and obtain a gradient bound for $u_{n,t}$ which is
uniform in n and t.

In the construction above, we have assumed that the injectivity radius
of $M$ is bounded from below by $1$. In the case that there is no
positive lower bound for the injectivity radius of $M$, we proceed as
follows. As pointed out earlier, $\text{exp}_p:T_pM\supset B_\pi(0)
\rightarrow B_\pi(p)$ is a local diffeomorphism. Thus we can pull back
the metric of $M$ to $B_\pi(0) \subset T_pM$. It is then easy to see
that $\text{exp}_p:T_pM\supset B_2(0) \rightarrow B_2(p)$ is a local
Riemannian covering map, and the injectivity radius at $0$ of
$B_\pi(0)\subset T_pM$ is $\pi$.  To obtain the gradient bounds at
$p_0 \in \partial A_n$ as discussed above, we can lift the whole
construction, including $A_n$ and $S_{n,t}$ locally to $B_1(0)\subset
T_pM$ and again use $Z_{r_0,p}$ to obtain the same gradient bound for
the lift of $u_{n,t}$. But this implies the gradient bound for
$u_{n,t}$ itself. This gives

\begin{lemma}
  \label{lem3.3}
  For every $0\leq t\leq \delta_0$ the surfaces $S_{n,t}$ exist and
  are smooth graphs of $u_{n,t}$ over $\ol{A}_n$ satisfying
\begin{align}
 \label{eq3.50} &  \, 0< u_{n,t} <t \hspace{.1in}\mbox{in $A_n$}\\
 \label{eq3.60} & \,|\nabla^M u_{n,t}| \leq C_3 \hspace{.1in} \mbox{on $\ol{A}_n$ }
\end{align}
for all $0\leq t \leq \delta_0$ and $n \in \N$, with $C_3$ independent
of $n$ and $t$.
\end{lemma}

\begin{proof}
  By comparing with planes of constant height zero and $\delta_0$, the
  height of the surfaces $S_{n,t}$ is bounded from below by zero and
  from above by $\delta_0$. The above construction of barriers at the
  boundary implies that
  $$|\nabla^M u_{n,t}| \leq |\nabla^M w|=\phi'(2r_0)=C_2$$
  on $\partial A_n$, independent of $n$ and $t$. By Lemma
  \ref{lem3.1}, this implies the stated a priori gradient bound for
  $u_{n,t}$ on $\ol{A}_n$. The DeGiorgi-Nash-Moser and Schauder
  estimates then imply a priori bounds of all higher derivatives of
  $u_{n,t}$ on $\ol{A}_n$. Thus we obtain existence by the method of
  continuity.
 \end{proof}

\begin{remark}
  Note that to get the existence of the surfaces $S_{n,t},\, 0\leq
  t\leq \delta$ just for an implicit $0<\delta\leq \delta_0$, one can
  argue that by the stability of $S_{1,0}$, the graphs $S_{1,t}$ exist
  for $t \in [0,\delta]$ and have bounded gradient. One can then use
  $S_{1,\delta}$ as an upper barrier for the surfaces $S_{n,t}$ on the
  inner boundary $\partial_1$ to obtain an a priori gradient estimate
  there.
\end{remark}

By construction, we have that $S_{n,t}$ lies above $S_{m,t}$ on $A_n$
for $m > n$. Since for $0\leq t \leq \delta_0$ the surfaces have
uniform gradient bounds, the DeGiorgi-Nash-Moser and Schauder
estimates imply locally uniform estimates for all higher
derivatives. We fix $t\in (0,\delta_0]$ and take the limit
$n\rightarrow \infty$ of the surfaces $S_{n,t}$ to obtain a limit
surface S, which is a minimal graph over $M\setminus D_1$ and has
boundary value $0$ on $\partial_1$.  Furthermore, the height function
$u$ is bounded by $\delta_0$ and the gradient of $u$ by $C$.

Since the gradient of $u$ is bounded, $S=\text{graph}(u)$ is
quasi-isometric to $M\setminus D_1$, hence it is parabolic. Thus the
height function $u$ on $S$ is a bounded harmonic function on
$\text{graph}(u)$ and so must be constant, equal to zero. That is $u
\equiv 0$ and the graphs $S_{n,t}$ converge locally uniformly to zero.

Now we can prove the half-space theorem.

\begin{proof}[Proof of Theorem \ref{th1.1}]
  Suppose $S$ is a minimal hypersurface properly immersed in $M \times
  (-\infty,c)$. Lowering $M\times \{ c\}$ until it ``touches'' $S$, we
  can suppose $S$ is asymptotic to $M\times \{c\}$ at infinity. More
  precisely, if $M\times \{\tau\}$ touches $S$ for the first time at
  some point of $S$ then $S = M\times \{\tau\}$ by the maximum
  principle and we are done. Otherwise the first contact is at
  infinity so we can assume $S$ is asymptotic to $M \times \{c\}$. By
  translating $S$ vertically we can assume that $c=0$.

  Since $S$ is proper, we can assume that there is a point $p_0\in M$
  and a cylinder $C= B_{r_0}(p_0)\times (-r_0,0)$ for some $r_0>0$
  such that $S \cap C =0$. We can assume that $r_0$ is less than the
  injectivity radius at $p$.  In our construction of the surfaces
  $S_{n,t_0}$, we choose $D_1 = B_{r_0/2}(p_0)$ and $t_0 =
  \min\{\delta_0, r_0/2\}$. Note that translating $S_{n,t_0}$
  vertically downwards by an amount $t_0$ keeps the boundaries of the
  translates of $S_{n,t_0}$ strictly above $S$. Thus by the maximum
  principle all the translates remain disjoint from $S$. We call
  $S_{n,t_0}^\prime$ this final translate. Note that all the surfaces
  $S_{n,t_0}^\prime$ lie above $S$ and converge as $n\rightarrow
  \infty$ to the plane $M\times \{-t_0\}$. Thus $S$ lies below
  $M\times \{-t_0\}$ which contradicts that $S$ is asymptotic to $M
  \times \{0\}$.
\end{proof}

\section{The graphical case}
\label{sec4}

\begin{theorem} \label{th1} Assume $M$ is complete with nonnegative
  Ricci curvature and sectional curvatures $K_{\pi} \geq -K_0$ for a
  nonnegative constant $K_0$.  Let $S=\text{graph}(u)$ be a minimal
  graph in $M \times \R $ over $B_R(p)$ with $u\geq 0$. Then
  \[ |\nabla^M u(p)| \leq C_1 e^ {C_2 u^2(p)\frac{\Psi(R)}{R^2 }} \]
  where $\Psi(R)=(n-1)\sqrt{K_0}R\coth{(\sqrt{K_0}R)}+1$.
\end{theorem}

\begin{proof} Let $h(x)=\eta(x) W(x)$ with $W=\sqrt{1+|\nabla^M u|^2}
  ,\,\, \eta(x)=g(\phi(x)),\,\, g(t)=e^{Kt}-1,\,\,
  \phi(x)=(-u(x)/2u(p)+(1-\frac{d(x,p)^2}{R^2}))^+$ where $+$ denotes
  the positive part.  Let $C(p)$ denote the cut locus of $p$ and
  $\mathcal{U}(p)=B_R(p) \setminus C(p)$ be the set of points $q\neq
  p$ in $B_R(p)$ for which there is a unique minimal geodesic $\gamma
  $ joining $p$ and $q$ with $q$ not conjugate to $p$ along $\gamma$.
  It is well-known that $d(x,p)$ is smooth on $\mathcal{U}(p)$ which
  is open. Note that
  $d(x,p)^2$ and so $h(x)$ is smooth in a  neighborhood of $p$.\\[2ex]
  \emph{Case 1:} The max of h
  occurs at a point  $q\in \mathcal{U}(p)$\\[2ex]
  From \eqref{eq13.1} we find since $M$ has nonnegative Ricci
  curvature, \be \label{eq15} Lh := \Delta^S h-2g^{ij}\frac{D_i
    W}{W}D_j h \geq Ke^{K\phi}W\big(\Delta_S \phi +K g^{ij}D_i \phi
  D_j \phi\big)\, .  \ee

  The point is now to choose $K$ so that $\Delta^S \phi +K g^{ij}D_i
  \phi D_j \phi>0$ on the set where $h>0$ and $W$ is large. We will
  need a standard comparison lemma \cite{K}.

\begin{lemma} \label{lem1} Suppose $M$ has sectional curvatures
    $K_{\pi} \geq -K_0$ for a nonnegative constant $K_0$. Let $q \in
    \mathcal{U}(p)$ Then the (nonzero) eigenvalues of $D^2 d(p, x)$ at
    $q$ (principal curvatures of the local distance sphere through
    $q$) are bounded above by those of the corresponding distance
    sphere in the hyperbolic space of curvature $-K_0$.
\end{lemma}

We have $\Delta^S u=0$ so
\[\Delta^S \phi=-\frac2{R^2}\big(d(x,p)\Delta^S d(x,p)+g^{ij} D_i d(x,p)D_j d(x,p)\big)\,.\]
Using Lemma \ref{lem1} and \eqref{eq12} we see that \be \label{eq16}
\Delta^S \phi \geq -\frac{\Psi(R)}{R^2} \ee where
$\Psi(R)=(n-1)\sqrt{K_0}R\coth\big(\sqrt{K_0}R\big)+1$
at a point $q \in \mathcal{U}(p)$.\\[2ex]
We next compute
\[ g^{ij}D_i \phi D_j \phi=g^{ij}D_i\bigg(\frac{u(x)}{2u(p)}+\frac{2d(x,p)}{R^2}D_i d(x,p)\!\bigg)D_j\bigg(\frac{u(x)}{2u(p)}+\frac{2d(x,p)}{R^2}D_j d(x,p)\!\bigg)\]
\[=\frac{|\nabla u|^2}{4u(p)^2 W^2}+\frac{4d^2(x,p)}{R^4}\bigg(1-\Big\langle \frac{\nabla u}W, \nabla d(x,p)\Big\rangle_{\!\!M}^2\bigg)
+\frac{2d(x,p)}{u(p)R^2} \frac{\big\langle \nabla u, \nabla d(x,p)\big\rangle_M}{W^2}~.\]\\[1ex]
Hence 
\be \label{eq18}
g^{ij}D_i \phi D_j \phi \geq \bigg(\frac{|\nabla u|}{2u(p)
  W}-\frac{2}{RW}\bigg)^2\, .
\ee
Now assume that 
\be \label{eq20}
W(q) \geq \max\bigg\{\frac2{\sqrt{3}}, \frac{16 u(p)}R \bigg\}\, .
\ee
Then from \eqref{eq18} and \eqref{eq20},
\be \label{eq22}
g^{ij}D_i \phi D_j \phi \geq \frac1{64 u(p)^2}~.
\ee
Therefore from \eqref{eq16}, \eqref{eq22},
\be \label{eq24}
 \big(\Delta^S \phi +K g^{ij}D_i \phi D_j \phi\big) (q) \geq -\frac{\Psi(R)}{R^2}+\frac{K}{64 u^2(p)}
\ee
Now choose 
\be \label{eq26}
K=64u^2(p) \frac{\Psi(R) }{R^2}+2~.
\ee
Then \eqref{eq16} and \eqref{eq24} imply $Lh(q)>0$ contradicting the maximum principle. Hence 
\eqref{eq20} cannot hold and so
\be \label{eq28}
W(q) \leq \max\bigg\{\frac2{\sqrt{3}}, \frac{16 u(p)}R \bigg\}~.
\ee

Therefore $h(p)=(e^{\frac{K}2}-1)W(p) \leq (e^K -1)\max\big\{\frac2{\sqrt{3}}, \frac{16 u(p)}R \big\}$. After some manipulation we see that Theorem \ref{th1} follows.\\[2ex]
\emph{Case 2:} $q \not \in \mathcal{U}(p)$.

\begin{lemma}\label{lem2} a) Suppose the maximum of $h$ in $B_R(p)$ occurs at $q$. Then there is a unique minimal unit speed geodesic $\gamma(s)$ joining $p$ and $q$.\\
  b) For any $\e>0$, let $p^{\e}= \gamma(\e)$. Then $d(x, p^{\e})$ is
  smooth in a neighborhood of $q$.
\end{lemma}

\begin{proof} a) Suppose the maximum of $h$ occurs at $q\neq p$.  Then
  since $h(x) \leq h(q)$ and
\[d(x,p)=R\sqrt{1-\frac{u(x)}{2u(p)}-\frac1K \log{\Big(1+\frac{h(x)}{W(x)}\Big)}} \,\,~,\]
we see that
\[d(x,p) \geq \psi(x):=R\sqrt{1-\frac{u(x)}{2u(P)}-\frac1K \log {\Big(1+\frac{h(Q)}{W(x)}\Big)}}\]
with equality at $q$. Note that $\psi(x)$ is possibly only well
defined locally in a small neighborhood $B_{2\rho}(q)$. In this case we can let $\ol{\psi}(x)=\lambda(x) \psi(x)$
where $ 0 \leq \lambda(x) \leq 1$ is a smooth cutoff function with
\[\lambda(x)= \left \{ \begin{array}{cl} 1 & x \in B_{\rho}(q) \\
                               0 & x\in B_R(p) \setminus B_{2\rho}(q) \\
 \end{array} \right.
  \]  
Then $d(x,p) \geq \ol{\psi}(x)$ in $B_R(p)$ with equality at $q$.

Hence we may assume $\psi(x)$ is smooth on $B_R(p) $ and so
\[\psi(x)-\psi(q) \leq d(x,p)-d(q,p) \leq d(x,q)~;\]
hence $ |\nabla^M \psi(q)| \leq 1$.

Now let $\gamma(s)$ be a unit speed minimal geodesic
joining $p$ to $q$. Then
\[\psi(\gamma(s))\leq s~.\]
and
\[\psi(\gamma(d(q,p))=\psi(q)=d(q,p)~.\]
Hence $\nabla^M \psi(q)=\gamma'(d(q,p))$
and so there is only one minimal geodesic joining
$p$ and $q$.\\[2ex]
b) Clearly $q$ is not conjugate to $p^{\e}$. Moreover since $d(x, p^{\e})+\e \geq d(x,p) \geq \psi(x)$ with equality at $q$, 
the argument of part a) shows that $\gamma$ is the unique minimal geodesic joining $p^{\e}$ and $q$. Hence $q \in \mathcal{U}(p^{\e})$ so  $d(x, p^{\e})$ is smooth in a neighborhood of $q$.
\end{proof}

We now complete the proof of case 2. Define 
\[\phi^{\e}=-\frac{u(x)}{2u(p)}+\bigg(1-\frac{(d(x,p^{\e})+\e)^2}{R^2}\bigg)^+,\, 
\eta^{\e}=g(\phi^{\e}),\,h^{\e}=\eta^{\e}W~.\]
Then since $d(x, p^{\e})+\e \geq d(x,p) \geq \psi(x)$ with equality at $q$, we have that 
\[\phi^{\e} \leq \phi,\, \eta^{\e} \leq \eta, h^{\e} \leq h~.\]
with equality at $q  \in \mathcal{U}(p^{\e})$. Thus by Lemma \ref{lem2} we may apply  case 1 (and let $\e \goto 0$) to complete the proof.
\end{proof}

\begin{corollary} \label{cor2} Let $M$ be as in Theorem \ref{th1}. If
  $S$ is a complete minimal graph with height function $ u\geq 0$,
  then $|\nabla^M u| \leq C_1$.
\end{corollary}
\begin{proof} Let $R \goto \infty$ in Theorem \ref{th1}.
\end{proof}

\begin{remark} 
  As in \cite{SP}, there is a version of Theorem \ref{th1} for graphs
  with constant or variable mean curvature $H(x)$ assuming the
  sectional curvatures of M are bounded below with no assumption on
  Ricci curvature. In particular Corollary \ref{cor2} holds for
  bounded solutions under these hypotheses. The method presented here
  sharpens the result of \cite{SP} in that no control of injectivity
  radius is needed.
\end{remark}

Set $m(R)=\inf_{B_R(P)} u$.  Then more generally we have
\begin{corollary}\label{cor3}
  Let $M$ be as in Theorem \ref{th1} and let $S$  is a complete minimal graph with height function u.\\
  a) If $K_0>0$ assume $\limsup_{R\goto \infty} \frac{m^2(R)}R=0$.\\
  b)  if $K_0=0$ assume $\limsup_{R\goto \infty} \frac{|m(R)|}R=0$.\\[1ex]
  Then $|\nabla^M u| \leq C_1$.
\end{corollary}

We can now use the Moser technique as developed by Saloff-Coste
\cite{SL} and Grigor'yan \cite{Gr1} to prove Theorem \ref{th1.2}\,,
which is an extension of the corresponding result of Bombieri, De
Giorgi and Miranda \cite{BGM} for $M=\R^n$.

Assume that $S$ is an entire minimal graph with height function $u\geq
0$. According to Corollary \ref{cor2}, $ |\nabla^M u| \leq C_1$
globally on $M$. Thus the induced metric $g_{ij}$ given by (\ref{eq5})
is uniformly elliptic and the Laplacian $\Delta^S$ on $S$ given by
(\ref{eq12}) is a divergence form uniformly elliptic operator.  We may
by translation assume $\inf_{M}u=0$. Thus given any $\e>0$ there is a
point $p\in M$ with $u(p) \leq \e$. Applying the Harnack inequality
Theorem 7.4 of \cite{SL} yields for all $R$
\[ \sup_{B_R(p)} u \leq C \inf_{B_R(p) } u \leq C\e~\]
for a uniform constant $C$ independent of $R$. Letting $R \goto \infty$ and then $\e \goto 0$ gives $u \equiv 0$.

\begin{remark}\label{rem2} Theorem \ref{th1.2}  can be improved somewhat to allow  $\limsup_{R\goto \infty} \frac{|m(R)|}{R^\alpha}=0$
  for some controlled small $\alpha \in (0,\frac12)$.
\end{remark}

\def\cprime{$'$}

\end{document}